\documentclass{article}

\usepackage{microtype}
\usepackage{graphicx}
\usepackage{subfigure}
\usepackage{booktabs} 
\usepackage{adjustbox}
\RequirePackage{amsthm,amsmath}
\RequirePackage{natbib}

\theoremstyle{plain}

\usepackage{float}

\usepackage[T1]{fontenc}
\usepackage{graphicx}
\usepackage{floatflt}
\usepackage{algorithm}
\usepackage{bm}
\usepackage{natbib}
\usepackage{rotating}
\usepackage{multirow}
\usepackage{booktabs}
\usepackage{xspace}
\usepackage{color}
\usepackage{braket}
\usepackage{algpseudocode}
\usepackage{tabularx}
\usepackage{pdfcomment}
\usepackage{amsmath}
\usepackage{amssymb}
\usepackage{amsthm}
\usepackage{bm}
\usepackage{rotating}
\usepackage{multirow}
\usepackage{booktabs}
\usepackage{placeins}
\usepackage{wrapfig}
\usepackage{graphicx}
\graphicspath{{./figures/}}
\usepackage{enumerate}
\usepackage{mathtools}
\usepackage{cleveref}
\usepackage{verbatim}

\usepackage{lipsum}
\usepackage{todonotes}
\usepackage{standalone}
\usepackage{url}

\newcommand{\indep}{\rotatebox[origin=c]{90}{$\models$}}

\DeclarePairedDelimiterX{\ExpArg}[1]{[}{]}{#1}

\theoremstyle{plain}
\newtheorem{theorem}{Theorem}[section]
\newtheorem{lemma}[theorem]{Lemma}

\newtheorem{proposition}[theorem]{Proposition}

\theoremstyle{definition}
\newtheorem{definition}{Definition}[section]

\theoremstyle{remark}

\definecolor{lightgray}{gray}{0.95}

\newcommand{\NormS}[1]{\mbox{}\left\|#1\right\|^2}

\newcommand{\setlinespacing}[1]%
           {\setlength{\baselineskip}{#1 \defbaselineskip}}

\newcommand{\BlockDiagk}[1]{\mbox{}\left(%
\begin{array}{cc}
  \Sigma_{k} & \bf{0} \\
  \bf{0} &  \Sigma_{\rho-k}\\
\end{array}\right)}

\newcommand{\BlockDiagkk}[1]{\mbox{}\left(%
\begin{array}{cc}
  \Sigma_{k} & \bf{0} \\
  \bf{0} & \bf{0} \\
\end{array}\right)}

\newcommand{\BlockDiagkrk}[1]{\mbox{}\left(%
\begin{array}{cc}
  \bf{0} & \bf{0} \\
  \bf{0} & \Sigma_{\rho-k} \\
\end{array}\right)}

\newcommand{\BlockDiagkkh}[1]{\mbox{}\left(%
\begin{array}{c}
  \Sigma_{k} \\
  \bf{0} \\
\end{array}\right)}

\newcommand{\BlockDiagkrkh}[1]{\mbox{}\left(%
\begin{array}{c}
  \bf{0} \\
  \Sigma_{\rho-k} \\
\end{array}\right)}
\newenvironment{Proof}{\noindent {\em Proof:}}

\long\def\killtext#1{}

\errorcontextlines\maxdimen

\makeatletter
\usepackage{amssymb}
\usepackage[accepted]{icml2021}

\icmltitlerunning{Parallel Quasi-Concave Set Optimization: A new frontier that does not need submodularity}

\begin{document}

\twocolumn[
\icmltitle{Parallel Quasi-concave set optimization:\\ A new frontier that scales without needing submodularity}




\begin{icmlauthorlist}
\icmlauthor{Praneeth Vepakomma}{to}
\icmlauthor{Yulia Kempner}{goo}
\icmlauthor{Ramesh Raskar}{to}
\end{icmlauthorlist}

\icmlaffiliation{to}{Massachusetts Institute of Technology}
\icmlaffiliation{goo}{Holon Institute of Technology}


\vskip 0.3in
]




\begin{abstract}
Classes of set functions along with a choice of ground set are a bedrock to determine and develop corresponding variants of greedy algorithms to  obtain efficient solutions for combinatorial optimization problems. The class of approximate constrained submodular optimization has seen huge advances at the intersection of good computational efficiency, versatility and approximation guarantees while exact solutions for unconstrained submodular optimization are NP-hard. What is an alternative to situations when submodularity does not hold? Can efficient and globally exact solutions be obtained? We introduce one such new frontier: The class of quasi-concave set functions induced as a dual class to monotone linkage functions. We provide a parallel  algorithm with a time complexity over $n$ processors of $\mathcal{O}(n^2g) +\mathcal{O}(\log{\log{n}})$ where $n$ is the cardinality of the ground set and $g$ is the complexity to compute the monotone linkage function that induces a corresponding quasi-concave set function via a duality. The complexity reduces to  $\mathcal{O}(gn\log(n))$ on $n^2$ processors and to $\mathcal{O}(gn)$ on $n^3$ processors. Our algorithm provides a globally optimal solution to a maxi-min problem as opposed to submodular optimization which is approximate. We show a potential for widespread applications via an example of diverse feature subset selection with exact global maxi-min guarantees upon showing that a statistical dependency measure called distance correlation can be used to induce a quasi-concave set function.
\end{abstract}
\section{Introduction}
The rich structure of some set function classes allows for development of efficient algorithms for combinatorial optimization problem.
 To be formal, a set system $(F, \mathcal{Z})$ is a collection $F$ of subsets of a ground set $\mathcal{Z}$. For example $F$ could be subsets of the power set of $\mathcal{Z}$ or could be subsets that satisfy the structure of a greedoid \cite{korte2012greedoids}, semi-lattice \cite{chajda2007semilattice}, independence systems\cite{conforti1989geometric} or an antimatroid\cite{dietrich1989matroids,kempner2003correspondence,algaba2004cooperative} and so forth. \par Popular set function classes such as submodular functions \cite{lovasz1983submodular,edmonds2003submodular,nemhauser1978analysis,fujishige2005submodular,feige2011maximizing,krause2014submodular,iyer2013submodular} have resulted in a wide array of powerful algorithms for several tasks across different fields. \par Under lack of submodularity, relaxations that characterize approximate submodularity, \cite{bian2017guarantees,bogunovic2018robust,horel2016maximization,chierichetti2020additive,das2018approximate} have been introduced to develop  combinatorial algorithms with approximation guarantees.  Other set function classes beyond submodularity include those of subadditive functions, quasi-submodular functions and the lesser known class of induced quasi-concave set functions that is relevant to this paper. \par This paper introduces a parallel algorithm for optimizing quasi-concave set functions with global optimality guarantees as opposed to submodular optimization that provides approximate solutions. Algorithms for optimizing general quasi-concave set functions do not exist, while a specific sub class of quasi-concave set functions that can be written in terms of monotone linkage functions can be optimized to obtain globally optimal solutions. As an example, we show that certain monotone linkage functions of distance covariance  induce a corresponding quasi-concave set function. We use our algorithm to find an optimally diverse set of features based on distance covariance. 
\subsection{Preliminaries}
We now list the definition of \textit{quasi-concave set functions} and state the \textit{induced quasi-concave set function optimization problem} which are central to the focus of this paper.

\section{Quasi-concave set functions}
\begin{definition}[\textbf{Quasi-Concave Set Function} \cite{mullat1976extremal,kuznecov1985monotonic,zaks1989incomplete,vepakomma2019diverse}]\label{qcvxDef}
A function $F : \mathcal{F} \mapsto \mathbb{R}$ defined on a set system $(\mathbf{X}, \mathcal{F})$ is quasi-concave
if for each $\mathbf{S, T}\in \mathcal{F}$, \begin{equation}\label{mon_func}
F(\mathbf{S} \cap \mathbf{T} ) \geq \min{\{F(\mathbf{S}), F(\mathbf{T})\}}\end{equation}
\end{definition}

\textbf{Connection:} We would like to note its notational similarity to its continuous counter-part of strictly quasi-concave functions which are those real-valued functions defined on any convex subset of real-valued vector spaces such that
${\displaystyle f(\lambda x+(1-\lambda )y)\geq \min {\big \{}f(x),f(y){\big \}}}$ for all ${\displaystyle x\neq y}$ and ${\displaystyle \lambda \in (0,1)}$.

We denote the set $2^\mathbf{X} \setminus \left\{\phi,\mathbf{X}\right\}$  by $\mathcal{P}^{-X}$ and we use $i$ indexed subsets like $S_i$ to indicate a singleton (unit cardinality) element of $\mathbf{S}$ labeled by $i$. 

\begin{definition}[\textbf{Monotone Linkage Function} \cite{mullat1976extremal}]

A function $\pi(X_i,\mathbf{Z})$ defined on $\mathbf{Z} \in \mathcal{P}^{-X}  , X_i \in \mathbf{X} \setminus \mathbf{Z}$ is called a monotone linkage function if \begin{equation}\pi(X_i, \mathbf{S}) \geq \pi(X_i, \mathbf{T}), \mathbf{S} \subseteq \mathbf{T}\in \mathcal{F},  \forall X_i \in \mathbf{X} \setminus T\end{equation}
\end{definition}

We would like to note for the clarity of the reader that $X_i$ is an element while $\mathbf{S},\mathbf{T}$ are sets. Therefore, to make this distinction clear we denote sets in bold-faced font and elements otherwise.

Monotone linkage functions have been introduced and used for clustering in \cite{kempner1997monotone,kempner2003clustering}.  A recent work \cite{seiffarth2021maximum} uses these functions to find maximum margin separations in finite closure systems.

\textbf{Induced quasi-concave set function optimization} This is stated as the problem of maximizing a quasi-concave set function  $M_{\pi}(\mathbf{T})$ over the modified power set  $\mathcal{P}^{-X}$: 
\begin{equation}\label{minFeqn10}
   \underset{\mathbf{T} \subset \mathbf{\mathcal{P}^{-X}}} {\mathrm{arg\enskip max }} \enskip M_{\pi}(\mathbf{T}) = \underset{\mathbf{T} \subset \mathbf{\mathcal{P}^{-X}}} {\mathrm{arg\enskip max }} \enskip \underset{X_i \in \mathbf{X}\setminus \mathbf{T}}{\text{min}}
  \pi(X_i,\mathbf{T})
\end{equation}
where $\pi(X_i,\mathbf{Z})$ is a monotone linkage function. 

\section{Contributions}
\begin{enumerate}
    \item We provide a parallel algorithm to find all the subsets that globally optimize the induced quasi-concave set function optimization problem in (\ref{minFeqn10}).

\begin{table*}[]
\begin{adjustbox}{width=\linewidth,center}
\begin{tabular}{|c|c|c|c|c|c|c|c|c|}
\hline
Type       & \begin{tabular}[c]{@{}c@{}}Induced Quasi-concave\\  set function\\ (Parallel: \textbf{Ours})\end{tabular} & \begin{tabular}[c]{@{}c@{}}Induced Quasi-concave\\ set function\end{tabular}         & \begin{tabular}[c]{@{}c@{}}Quasi-concave\\ set function \\ (General purpose)\end{tabular} & \begin{tabular}[c]{@{}c@{}}Unconstrained\\ Submodular\end{tabular} & \begin{tabular}[c]{@{}c@{}}Robust \\ submodular\end{tabular} & \begin{tabular}[c]{@{}c@{}}Unconstrained \\ Quasi \\ submodular\end{tabular} & \begin{tabular}[c]{@{}c@{}}Quasi semistrictly submodular\\ M-/L-convex\end{tabular} & \begin{tabular}[c]{@{}c@{}}$SSQM^{\neq}$\\ under M-convex \\ domain\end{tabular} \\ \hline
Complexity & \begin{tabular}[c]{@{}c@{}}On $n$ processors,\\ $\mathcal{O}(n^2g) +\mathcal{O}(\log{\log{n}})$. \\For $n^2, n^3$ processors, \\check Table 2. \end{tabular}       & \begin{tabular}[c]{@{}c@{}}$\mathcal{O}(n^3g)  +\mathcal{O}(n)$\end{tabular} & Unknown                                                                                    & NP-Hard                                                            & $\mathcal{O}(nk)$                                                      & $\mathcal{O}(n^2)$                                                                     & \begin{tabular}[c]{@{}c@{}}$\mathcal{O}(n^2\log{L})  + \mathcal{O}(n^2)$\end{tabular}                 & \begin{tabular}[c]{@{}c@{}}$\mathcal{O}(n^4(\log{L})^2)$\\ \end{tabular}                    \\ \hline
Solution   & \begin{tabular}[c]{@{}c@{}}Globally\\ optimal\end{tabular}                                 & \begin{tabular}[c]{@{}c@{}}Globally\\ optimal\end{tabular}                           & Unknown                                                                                    & Unknown                                                            & Approximate                                                  & Approximate                                                                  & Approximate                                                                         & Approximate                                                                       \\ \hline
\end{tabular}

\end{adjustbox}
\caption{We show the computational complexity of our parallel algorithm and contrast it with that of its non-parallel version (cubic complexity), settings of submodular optimization and its relaxations. $n$ is the size of the ground set, $k$ is the cardinality of the returned set = $\max\left\{|x(v) - y(v)||x,y \in dom\; f, v\in V\right\}$ where $f:Z^V \mapsto \mathbb{R} \cup \left\{+\infty\right\}$ and $g$ is the complexity to compute the monotone linkage function.}
\end{table*}

\item The proposed parallel algorithm has a time complexity over $n$ processors of $\mathcal{O}(n^2g) +\mathcal{O}(\log{\log{n}})$ where $n$ is the cardinality of the ground set and $g$ is the complexity to compute the monotone linkage function that induces a corresponding quasi-concave set function via a duality. The complexity reduces to  $\mathcal{O}(gn\log(n))$ on $n^2$ processors and to $\mathcal{O}(gn)$ on $n^3$ processors. The parallel approach reduces the currently existing cubic computational complexity of the non parallel version which is $\mathcal{O}(n^3g) + \mathcal{O}(n)$.

\item As an example, we show that some functions of distance covariance (a measure of statistical dependence) are quasi-concave set functions. This lets us optimize them to obtain globally optimal maxi-min solutions for the most diverse subset of features. 


\end{enumerate}

    \subsection{Quasi-concave set function optimization under various set systems}
A greedy-type algorithm for finding maximizers of induced quasi-concave set functions was constructed in \cite{mullat1976extremal,kuznecov1985monotonic,zaks1989incomplete}. Inspired by this work, extensions of these algorithms were developed for the setting of multipartite graphs in \cite{vashist2006multipartite}. Similarly, quasi-concave set functions of distance covariance were derived in \cite{vepakomma2019diverse} and their optimization resulted in a solution for a diverse feature selection problem with guarantees. 
Furthermore, quasi-concave set functions were extended to various set systems including antimatroids \citep{levit2004quasi} and meet-semilattices  in \cite{kempner2008quasi}.

\begin{table}[!htbp]
\centering
\begin{tabular}{|l|l|}
\hline
\multicolumn{1}{|c|}{\textbf{\begin{tabular}[c]{@{}c@{}}\# of \\ processors\end{tabular}}} & \multicolumn{1}{c|}{\textbf{\begin{tabular}[c]{@{}c@{}}Time \\ Complexity\end{tabular}}} \\ \hline
$n$  \;\;(Ours)                                                                                        & $\mathcal{O}(n^2g )$                                                                    \\ \hline
$n^2$ (Ours)                                                                                      & $\mathcal{O}(gn\log{n})$                                                                 \\ \hline
$n^3$ (Ours)                                                                               & $\mathcal{O}(gn)$                                                                        \\ \hline
Non-parallel                                                                               & $\mathcal{O}(n^{3}g)+\mathcal{O}(n)$                                         \\ \hline
\end{tabular}
\caption{In this table, we show the complexity of our proposed parallel algorithm with respect to increasing number of processors $n, n^2 \& n^3$. Here, $n$ is also chosen to be around the order of size of the ground set. We show that the running times can be drastically reduced from the cubic complexities in the non-parallel version.}
\end{table}

\section{Related work: Comparing quasi-concave set functions with submodularity} Given the seminal impact of submodular optimization, we would like to compare the definitions of quasi-concave set functions with submodular functions and their relaxations. We state some connections inline that we find accordingly.  
\begin{enumerate}
    
    \item \textbf{Submodular optimization} \cite{fujishige2005submodular}
    Let $\textbf{V}$ be a ground set with cardinality $|\textbf{V}|=n$, and let $f$ :
$2^{\textbf{V}} \rightarrow \mathbb{R}_{\geq 0}$ be a set function defined on $\textbf{V} .$ The function $f$ is said to be submodular if for any sets $\textbf{X} \subseteq \textbf{Y} \subseteq \textbf{V}$ and any element $e \in V \backslash Y$, it holds that the discrete derivative
$$f(\textbf{X} \cup\{e\})-f(\textbf{X}) \geq f(\textbf{Y} \cup\{e\})-f(\textbf{Y})$$  is non-increasing in  $\textbf{X}$. That is, the incremental gain of adding an element to a subset is $\geq$ (is not smaller) the incremental gain of adding it to a superset.  An equivalent definition is that for every ${\displaystyle \textbf{S,T}\subseteq \textbf{V} }$ we have that \begin{equation}
    {\displaystyle f(\textbf{S})+f(\textbf{T})\geq f(\textbf{S}\cup \textbf{T})+f(\textbf{S}\cap \textbf{T})}
\end{equation}
The problem of maximizing a normalized monotone submodular function subject to a cardinality constraint has been studied extensively. A celebrated result
of (Nemhauser et al., 1978) shows that a simple greedy algorithm that starts with an empty set and then iteratively adds elements with highest marginal gains provides a $(1 - 1/e)$-approximation.
\\\textbf{Connection:}
Upon defining a linkage function to be equal to a discrete  derivative of a submodular function as $$\pi(e,\textbf{X})=f(\textbf{X} \cup\{e\})-f(\textbf{X})$$ it can be seen that the derivative of a submodular function is a monotone linkage function. However, not every monotone linkage function is a derivative of some submodular function \cite{muchnik1987submodularI,muchnik1987submodularII}.  Combining equations (3) and (4), we can say that the functions that are both submodular and quasi-concave set functions would satisfy $f(\mathbf{S}) + f(\mathbf{T}) >= f(\mathbf{S} \cup \mathbf{T}) + f(\mathbf{S} \cap \mathbf{T}) >= f(\mathbf{S} \cup \mathbf{T}) + \min\left\{f(\mathbf{S}), f(\mathbf{T})\right\}$.

\item\textbf{Robust submodular optimization} Robust versions of submodular optimization problem  were introduced in \citep{krause2008robust,mirzasoleiman2017deletion,bogunovic2017robust,kazemi2018scalable,iyer2019unified,avdiukhin2019adversarially,powers2016constrained}. An earlier variant is of the form introduced in \cite{krause2008robust} as  $$
\max _{\textbf{S} \subseteq \textbf{V},|\textbf{S}| \leq k} \min _{\textbf{Z} \subseteq \textbf{S},|\textbf{Z}| \leq \tau} f(\textbf{S} \backslash \textbf{Z})
$$
The $\tau$ refers to a robustness parameter, representing the size of the subset $\textbf{Z}$ that is removed from the selected set $\textbf{S}$. The goal is to find a set $\textbf{S}$ such that it is robust upon the worst possible removal of $\tau$ elements, i.e., after the removal, the objective value should remain as large as possible. For $\tau=0$, the problem reduces to standard submodular optimization. The greedy algorithm, which is near-optimal for standard submodular optimization can perform arbitrarily badly for the robust version of the problem.
\\\textbf{Connection:} Note that our statement of induced quasi-concave set function optimization problem naturally has a robustness component that is similar to the max-min constraints used in the literature on robust submodular optimization. 
\item\textbf{Quasi submodular and semi-strictly submodular functions }\cite{mei2015unconstrained} A set function $F: 2^{N} \mapsto \mathbb{R}$ is quasi-submodular function if $\forall \textbf{X, Y} \subseteq \textbf{N}$, \textit{both} of the following conditions are satisfied
$$
\begin{array}{l}
F(\textbf{X} \cap \textbf{Y}) \geq F(\textbf{X}) \Rightarrow F(\textbf{Y}) \geq F(\textbf{X} \cup \textbf{Y}) \\
F(\textbf{X} \cap \textbf{Y})>F(\textbf{X}) \Rightarrow F(\textbf{Y})>F(\textbf{X} \cup \textbf{Y})
\end{array}
$$
On a similar note, a rich family of semistrictly submodular, discrete Quasi L-convex and discrete M-convex functions were introduced in \cite{murota1998discrete,murota2009recent}.

\end{enumerate}

\begin{figure}
    \centering
    \includegraphics[scale=0.7]{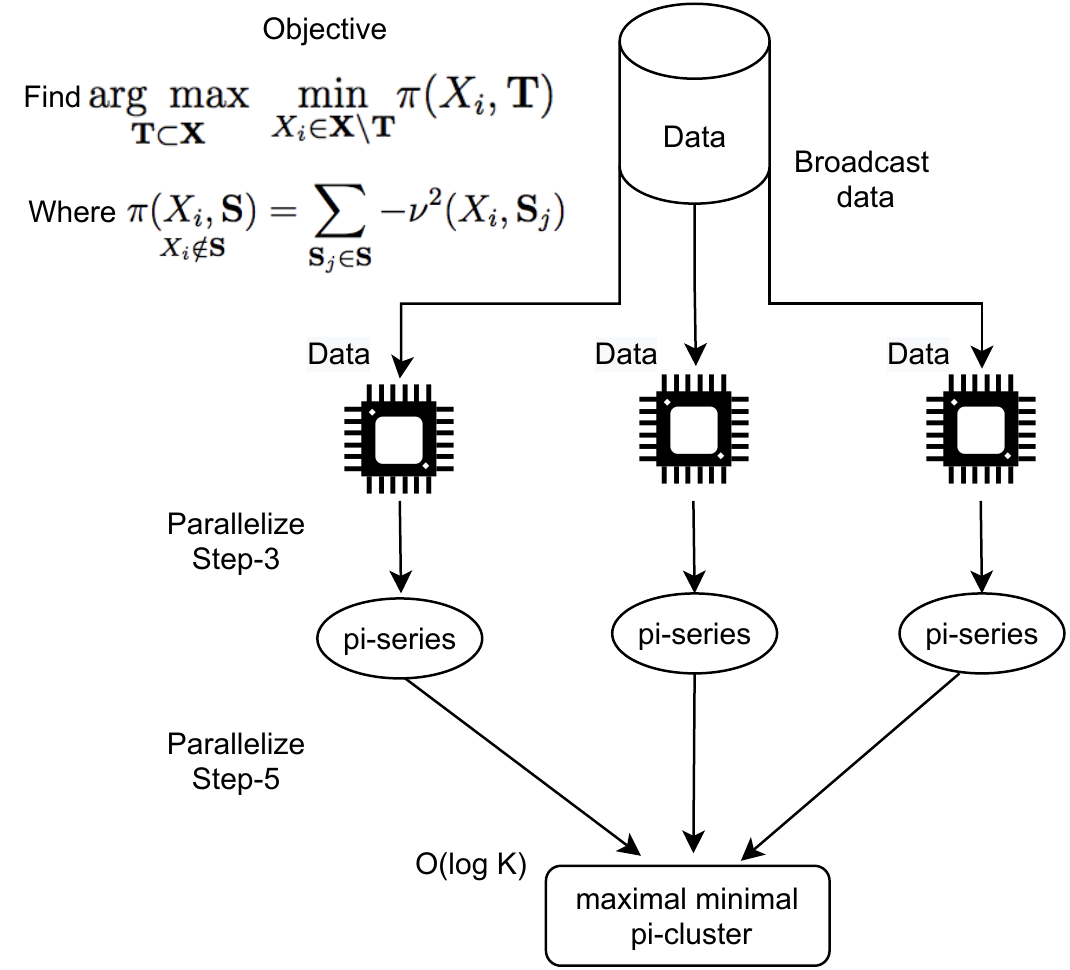}
    \caption{The proposed parallel algorithm consists of generating a $\pi$-series at each parallel entity over a copy of the data. The $\pi$-series at each entity starts with a different $X_i$. Each entity then generates a $\pi$-cluster corresponding to its generated $\pi-series$.The final step involves picking the best $\pi-cluster$. This is the only step that is not done in parallel.}
    \label{fig:my_label1}
\end{figure}
\section{Algorithm and proof of optimality} We now introduce required definitions and corresponding theory to derive the algorithm. This includes definitions for $\pi$-series and $\pi$-clusters
\begin{definition}[\textbf{$ \pi$-series}]
 We refer to a series $s_{\pi}=(X_{i_1},\ldots, X_{i_N})$ as a $\pi$-series if \begin{equation}\pi({X}_{i_{k+1}}, \bf{\overline{S}_k}) = \underset{X_i \in \mathbf{X}\setminus \mathbf{\overline{S}_k}}{\text{min}}
  \pi(\textnormal{X}_i,\mathbf{\overline{S}_k}) \end{equation}
  
  for any starting set $\bf{\overline{S}_k} = \{X_{i_1},\ldots,X_{i_k}\}, k = {1,\ldots,N-1}$.
 \end{definition}
 Therefore, it is a way of greedily populating a series that can start with any first element $\bf{X}_{i_1}$ being the current series, but the subsequent element to be added to the series, must be the element that minimizes the element to current series function of $\pi(\bf{X}_{i_{k+1}},\bf{\overline{S}_k})$ where $\bf{X}_{i_{k+1}}$ is the next element added and $\bf{\overline{S}_k}$ is the current series.
 \begin{definition}[\textbf{$\pi$-cluster}]
A subset $\bf{S}\in \mathcal{P}^{-\mathbf{X}}$ will be
referred to as a $\pi$-cluster if there exists a $\pi$-series, $s_\pi = (X_{i_1},\ldots,X_{i_N})$, such that $\bf{S}$ is a maximizer of $M_{\pi}(\bf{\overline{S}_k})$ over all starting sets $\bf{\overline{S}_k}$ of $s_\pi$.
 \end{definition}

\begin{theorem}\label{Theorem7.1}\cite{kempner1997monotone}
 If for a $\pi$-series $s_{\pi} = (X_{i_1},X_{i_2},\ldots,X_{i_N})$, a subset $\mathbf{S}\subset \mathbf{X}$ contains $X_{i_1}$, and if $X_{i_{k+1}}$ is the first element in $s_{\pi}$ not contained in $\mathbf{S}$ (for some $k \in  \{1,\ldots, N - 1\}$, then $M_{\pi}(\mathbf{\overline{S}_k}) \geq  M_{\pi}(\mathbf{S})$

where $\mathbf{\overline{S}_k} = \left(X_{i_1},\ldots, X_{i_k}\right )$. In particular, if $\mathbf{S}$ is an inclusion-minimal maximizer of $M_{\pi}$ (with regard
to $\mathcal{P}^{-\mathbf{X}})$, then $\mathbf{S} = \mathbf{\overline{S}_k}$, that is, $\mathbf{S}$ is a $\pi$-cluster.
\end{theorem}
  From \cite{kempner1997monotone} we have
\begin{proposition}
If $\bf{S_1},\bf{S_2} \subset \bf{X}$ are overlapping maximizers of a quasi-concave set function $M_\pi(\bf{S})$ over $\mathcal{P}^{-\bf{X}}$, then $\bf{S_1} \cap \bf{S_2}$ is also a maximizer of $M_\pi(\bf{S})$.
\end{proposition}
This means that the minimal maximizers of a quasi-convex set function are not overlapping. Moreover, any nonminimal maximizer can be uniquely partitioned into a set of the minimal ones.

\begin{theorem}\label{thref1}
Each maximizer of a quasi-concave set function on $\mathcal{P}^{-\bf{X}}$ is a union of its
inclusion-minimal maximizers.
\end{theorem}
\begin{proof}
 Indeed, if $\bf{S^\ast}$ is a maximizer of $M_\pi(\bf{S})$ over $\mathcal{P}^{-\bf{X}}$, then, according to Theorem \ref{Theorem7.1}, for
any $X_{i} \in \bf{S^\ast}$, there exists a minimal maximizer included in $\bf{S^\ast}$ and containing $X_{i}$.
\end{proof}

  \begin{algorithm}[H]
   \caption{Algorithm for induced quasi-convex set function optimization}
    \begin{algorithmic}[1]
      \Function{=DiverseMinimalMaximDCoV}{$\mathbf{X}$}
\ForAll{$X_i \in \bf{X}$} \\\enskip\enskip\enskip Greedily form $\pi$-series $s_\pi(x) = (X_i,X_{i_2}\ldots X_{i_N})$ starting from $X_i$ as its first \item[]\enskip\enskip\enskip element. 
      \enskip\enskip\enskip\hspace{5em} \For {each $\pi$-series $s_\pi(x)$ in step 3}\\ \enskip\enskip\enskip\enskip\enskip\enskip Find a corresponding smallest starting subset $\bf{T_x}$ with  $$M_\pi(\bf{T_x}) = \underset{ 1 \leq k \leq N-1}{\mathrm{max}} \pi(X_{i_{k+1}},\{X_{i_{1}},\ldots,X_{i_{k}}\})$$\EndFor \EndFor
    \State Among the non-coinciding minimal $\pi$-clusters $T_x$'s choose those that maximize $$ M_\pi(\bf{T_x})=\underset{X_i \in \mathbf{X}\setminus \mathbf{T_x}}{\text{min}}
  \pi(X_i,\mathbf{T_x}) $$
    \enskip\enskip\enskip all of which are the required minimal maximizers, and we return them as minimalMax\\

    \Return(minimalMax)
       \EndFunction

\end{algorithmic}
\end{algorithm}

\begin{theorem}
 The algorithm above finds all the minimal maximizers over $\mathcal{P}^{-\bf{X}}$.
\end{theorem}

\begin{proof}
 From Theorem \ref{thref1} it follows that each element of minimalMax is a maximizer of $M_\pi(\bf{S})$ over $\mathcal{P}^{-\bf{X}}$.
 Assume that there is a minimal maximizer $\bf{S}$ that does not belong to minimalMax, and let $X_{i} \in \bf{S}$. Then, according to Theorem \ref{Theorem7.1}, there exist $\pi$-series starting from $X_i$ and minimal $\pi$-cluster $T_x \subseteq \bf{S}$ containing $X_{i}$ with $M_\pi(\bf{T_x}) \geq M_\pi(\bf{S})$.
 Since $\bf{S}$ does not belong to minimalMax, and, according to Steps $5$ and $8$ of the algorithm, $T_x$ or some subset of $T_x$ belongs
 to minimalMax, there is a minimal maximizer strictly
 included in $\bf{S}$  which contradicts the minimality of $\bf{S}$.
\end{proof}

\section{Computational complexity}
When we have $n$ processors, then we can build each $\pi$-series (in step-3 of algorithm) in $\mathcal{O}(n^2g)$ on one processor (including step 5), and because we build them in parallel, steps 3-5 take $\mathcal{O}(n^2g)$ time. Finding the maximum in step 8 takes  $\mathcal{O}(\log{\log{n}})$ time on $n$ processors, under the CRCW (concurrent-read-concurrent-write) mode \cite{horowitz1978fundamentals,horiguchi1989parallel,valiant1975parallelism,krizanc1999survey}.
If we have $n^2$ processors, $n$ processors are used to build each $\pi$-series. To add one element to a series we have to find $\min$ between $n$ elements, that takes $\mathcal{O}(\log{\log{n}})$ on $n$ processors, so to build each pi-series takes 
$g*(\log{1}+\log{2}+\ldots+\log{n})=\mathcal{O}(gn\log{n})$, and to finish it we have to find $\max$ with $n^2$ processors which takes $\mathcal{O}(1)$ time. This gives us $\mathcal{O}(gnloglog n)$ complexity. If we have $n^3$ processors, then we can use $n^2$ processors to build each $\pi$-series. To add one element to a series we have to find $\min$ between $n$ elements which takes $\mathcal{O}(1)$ on $n^2$ processors. So to build each $\pi$-series takes $\mathcal{O}(gn)$ time, and to finish we have to find $\max$ with $n^3$ processors, that takes $\mathcal{O}(1)$ time. These are summarized in Tables 1 and 2.

\section{Maxi-min Diverse Variable Selection}

As an illustrating example, that we derive, we aim to find all the subsets that maximize the function $M_{\pi}(\mathbf{T})$ which result in the solutions which are diverse features in the context of statistics/machine learning as follows
\begin{equation}\label{minFeqn1}
   \underset{\mathbf{T} \subset \mathbf{X}} {\mathrm{arg\enskip max }} \enskip M_{\pi}(\mathbf{T}) = \underset{\mathbf{T} \subset \mathbf{X}} {\mathrm{arg\enskip max }} \enskip \underset{X_i \in \mathbf{X}\setminus \mathbf{T}}{\text{min}}
  \pi(X_i,\mathbf{T})
\end{equation}
For specificity, we use distance covariance upon normalization of the data as a measure of statistical dependence to model the diversity via $\pi(\mathbf{X_i,S})$ as defined in Lemma 8.1.
 \section{\textbf{Relevant Background on Distance Covariance and Distance Correlation}}
In this section we introduce some preliminaries about distance correlation and distance covariance and illustrate a connection between these functions and quasi-concave set function optimization. 
Distance Correlation \cite{szekely2007measuring} is a measure of nonlinear statistical dependencies between random vectors of arbitrary dimensions. We describe below distance covariance $\mathbb{\nu}^2(\mathbf{x},\mathbf{y})$ between random variables $\mathbf{x} \in \mathbb{R}^d$ and $\mathbf{y} \in \mathbb{R}^m$ with finite first moments is a non-negative number as
	\begin{equation}\label{charac}
		\mathbb{\nu}^2(\mathbf{x},\mathbf{y})=\int_{\mathbb{R}^{d+m}}|f_{\mathbf{x},\mathbf{y}}(t,s)-f_\mathbf{x}(t)f_\mathbf{y}(s)|^2 w(t,s)dtds
	\end{equation}

where $w(t, s)$ is a weight function as defined in \cite{szekely2007measuring}, $f_\mathbf{x},f_\mathbf{y}$ are characteristic functions of $\mathbf{x},\mathbf{y}$ and $f_{\mathbf{x},\mathbf{y}}$ is the joint characteristic function.\par The distance covariance is zero if and only if random variables $\mathbf{x}$ and $\mathbf{y}$ are independent. Using the above definition of distance covariance, we have the following expression for Distance Correlation \cite{szekely2007measuring}:\par
 The squared Distance Correlation between random variables $\mathbf{x} \in \mathbb{R}^d$ and $\mathbf{y} \in \mathbb{R}^m$ with finite first moments is a nonnegative number is defined as

\begin{equation} \rho^2(\mathbf{x},\mathbf{y})
	    =\left\{ \begin{array}{cc}
				 \frac{\mathbb{\nu}^2(\mathbf{x},\mathbf{y})}{\sqrt{\mathbb{\nu}^2(\mathbf{x},\mathbf{x})\mathbb{\nu}^2(\mathbf{y},\mathbf{y})}},
				 & \mathbb{\nu}^2(\mathbf{x},\mathbf{x})\mathbb{\nu}^2(\mathbf{y},\mathbf{y})>0.\\
	        0, & \mathbb{\nu}^2(\mathbf{x},\mathbf{x})\mathbb{\nu}^2(\mathbf{y},\mathbf{y})=0.
	        \end{array} \right.
	\end{equation}

The Distance Correlation defined above has the following interesting properties.

\begin{enumerate}
    \item ${\rho}^2(\mathbf{x},\mathbf{y})$	 is applicable for arbitrary dimensions $d$ and $m$ of $\mathbf{x}$ and $\mathbf{y}$ respectively.

    \item ${\rho}^2(\mathbf{x},\mathbf{y})=0$ if and only if $\mathbf{x}$ and $\mathbf{y}$ are independent.

    \item ${\rho}^2(\mathbf{x},\mathbf{y})$ satisfies the relation $0 \leq \rho^2(\mathbf{x},\mathbf{y}) \leq 1$.
\end{enumerate}
 \subsection{\textbf{Sample Distance Covariance and Sample Distance Correlation}}
 We provide the definition of sample version of distance covariance given samples $\{ (\mathbf{x}_k,\mathbf{y}_k) | k = 1,2,\ldots, n \}$ sampled i.i.d. from joint distribution of random vectors $\mathbf{x} \in \mathbb{R}^d$ and $\mathbf{y} \in \mathbb{R}^m$. To do so, we define two squared Euclidean distance matrices $\mathbf{E}_\mathbf{X}$ and $\mathbf{E}_\mathbf{Y}$,  where each entry $[\mathbf{E}_\mathbf{X}]_{k,l} = \NormS{\mathbf{x}_k-\mathbf{x}_l}$ and $[\mathbf{E}_\mathbf{Y}]_{k,l} = \NormS{\mathbf{y}_k-\mathbf{y}_l}$ with $k,l \in \{ 1,2,\ldots, n\}$. These squared distance matrices are double-centered by making their row and column sums zero and are denoted as $\widehat{\mathbf{E}}_{\mathbf{X}}, \widehat{\mathbf{Q}}_{\mathbf{X}}$, respectively. So given a double-centering matrix $\mathbf{J}=\mathbf{I}-\frac{1}{n}\mathbf{1}\mathbf{1}^T$, we have $\widehat{\mathbf{E}}_\mathbf{X}=\mathbf{J}\mathbf{E}_\mathbf{X}\mathbf{J}$ and $\widehat{\mathbf{E}}_\mathbf{Y}=\mathbf{J}\mathbf{E}_\mathbf{Y}\mathbf{J}$. The sample distance covariance and sample distance correlation can now be defined as follows.

\begin{definition}{\textbf{Sample Distance Covariance \cite{szekely2007measuring}:}}
Given i.i.d samples $\mathcal{X} \times \mathcal{Y} = \{ (\mathbf{x}_k,\mathbf{y}_k) | k = 1,2,3,\ldots, n\}$ and corresponding double centered Euclidean distance matrices $\widehat{\mathbf{E}}_\mathbf{X}$ and $\widehat{\mathbf{E}}_\mathbf{Y}$, the squared sample distance correlation is defined as,
\label{popDC}\[
    \hat{\mathbb{\nu}}^2(\mathbf{X},\mathbf{Y})=\frac{1}{n^2}\sum_{k,l=1}^{n}[\widehat{\mathbf{E}}_\mathbf{X}]_{k,l}[\widehat{\mathbf{E}}_\mathbf{Y}]_{k,l},
\]	
\end{definition}
Using this, sample distance correlation is given by
\label{sampleDC}
\[
	\hat{\rho}^2(\mathbf{X},\mathbf{Y})
	= \left\{ \begin{array}{cc}
    	 \frac{\mathbf{\hat{\nu}}^2(\mathbf{X},\mathbf{Y})}{\sqrt{\mathbf{\hat{\nu}}^2(\mathbf{X},\mathbf{X})\mathbf{\hat{\nu}}^2(\mathbf{Y},\mathbf{Y})}}, & \mathbf{\hat{\nu}}^2(\mathbf{X},\mathbf{X})\mathbf{\hat{\nu}}^2(\mathbf{Y},\mathbf{Y})>0. \\
    	0, & \mathbf{\hat{\nu}}^2(\mathbf{X},\mathbf{X})\mathbf{\hat{\nu}}^2(\mathbf{Y},\mathbf{Y})=0.
	\end{array}
	\right.
\]

 \textbf{Monotonicity of distance covariance under lack of independence:}
    If $\mathbf{X,Z} \in \mathbb{R}^p$ and $\mathbf{Y} \in \mathbb{R}^q$ and if $\mathbf{Z} \indep (\mathbf{X},\mathbf{Y})$ then \begin{equation}\nu^2(\mathbf{X}+ \mathbf{Z},\mathbf{Y}) \leq \nu^2(\mathbf{X},\mathbf{Y})\end{equation}
Note that $\indep$ indicates 'statistically independent' in statistical literature. 

\subsection{Motivating applications for modeling diversity with quasi-concave set function optimization}
A minor sampling of applications that benefit from the results in this paper do parallel traditional applications seen in submodular optimization literature. A few directions are listed below.
\begin{enumerate}
    \item Maximally/minimally correlated marginal selection for private data synthesis \cite{zhang2021privsyn}. 
    \item Modeling diversity in active learning \cite{wei2015submodularity}, determinantal point processes \cite{tschiatschek2016learning}.
    \item Diverse sample selection, feature selection and data summarization in machine learning and statistics. \cite{prasad2014submodular,das2012selecting}
\end{enumerate}

\begin{figure}
    \centering
    \includegraphics[scale=0.9]{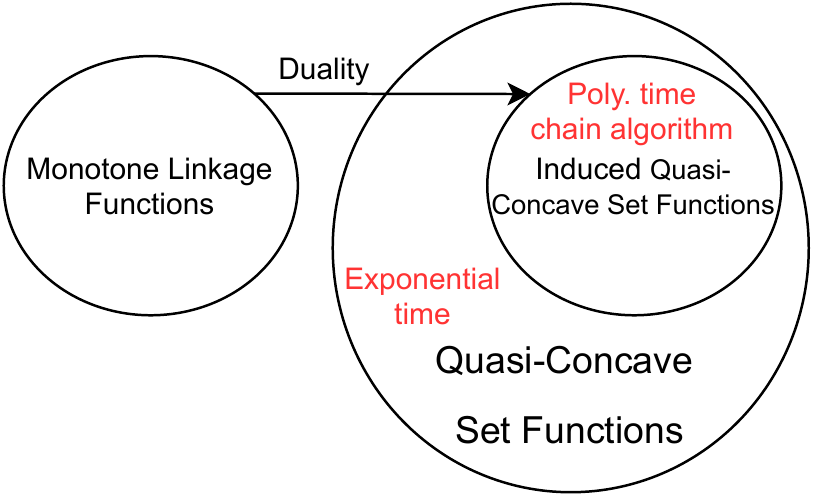}
    \caption{This illustration refers to the duality between monotone linkage functions and quasi-concave set functions. Optimization algorithms for general quqasi-concave set functions do not exist while those that are induced via monotone linkage functions can be optimized in polynomial time.}
    \label{fig:my_label}
\end{figure}

\subsection{\textbf{A monotone linkage function of distance covariance}}
\begin{lemma}
The function $\pi(X_i,\mathbf{S})$ of distance covariance defined on $X_i \notin \mathbf{S}$  as \begin{equation}\label{piEqn}
    \underset{X_i \notin \mathbf{S}}{\pi(X_i,\mathbf{S})} = \sum_{\mathbf{S}_j \in \mathbf{S}} -\nu^{2}(X_i,\mathbf{S}_j)
\end{equation}
 is a monotone linkage function. \end{lemma}\begin{Proof} For $\mathbf{S} \subseteq \mathbf{T}$ we have \begin{align}\underset{X_i\notin \mathbf{T}}{\pi(X_i,\mathbf{T})}&=\sum_{\mathbf{S}_j \in \mathbf{S}} -\nu_{i}^{2}(X_i,\mathbf{S}_j) -\sum_{\mathbf{T}_j \in \mathbf{T \setminus S}} \nu_{i}^{2}(X_i,\mathbf{T}_j) \\ &\leq
 \underset{X_i\notin \mathbf{T}  }
 {\pi(X_i,\mathbf{S})}=\sum_{\mathbf{S}_j \in \mathbf{S}} -\nu_{i}^{2}(X_i,\mathbf{S}_j)\end{align}We would also like to note that as $\nu(\cdot)$ is a non-negative function the above inequality does hold true.
 \end{Proof} $\\$By Assertion 1 from \cite{kempner1997monotone}, we conclude that
the function $M_{\pi}(\mathbf{T}) = \underset{X_i \in \mathbf{X}\setminus \mathbf{T}}{\text{min}}
  \pi(X_i,\mathbf{T})$ is a quasi-concave set function.


\begin{theorem}[Quasi-Concave Distance Covariance Set Function Theorem]\label{thm:mvt}


 If we have $\mathbf{S}\cap \mathbf{T} \neq \varnothing \text{ and } \forall \mathbf{S}, \mathbf{T}, \mathbf{Y} \text{ if } \nu^2(\mathbf{S}, \mathbf{T}) > 0 \land \nu^2(\mathbf{S}, \mathbf{Y}) > 0 \land \nu^2(\mathbf{T}, \mathbf{Y}) > 0  \text{ then, we have }$ \begin{equation}-\nu^2(\mathbf{S} \cap \mathbf{T}, \mathbf{Y}) \geq min(-\nu^2(\mathbf{S},\mathbf{Y}),-\nu^2(\mathbf{T},\mathbf{Y}))\end{equation}
\end{theorem}

\begin{proof}\cite{vepakomma2019diverse}


If $\mathbf{S} \cap \mathbf{T} = \mathbf{S}$ then since $\mathbf{S} \subseteq \mathbf{T}$


the Kosorok's distance covariance inequality simplifies to give \begin{equation}-\nu^2(\mathbf{S},\mathbf{Y}) \geq -\nu^2(\mathbf{T},\mathbf{Y})\end{equation}  Therefore, we have
\begin{equation}-\nu^2(\mathbf{S} \cap \mathbf{T}, \mathbf{Y}) \geq min(-\nu^2(\mathbf{S},\mathbf{Y}),-\nu^2(\mathbf{T},\mathbf{Y}))\nonumber\end{equation}
Similarly, if $\mathbf{S} \cap \mathbf{T} = \mathbf{T}$, then since $\mathbf{T} \subseteq \mathbf{S}$
 \begin{equation}-\nu^2(\mathbf{T},\mathbf{Y}) \geq -\nu^2(\mathbf{S},\mathbf{Y})\end{equation}and therefore,
\begin{equation}-\nu^2(\mathbf{S} \cap \mathbf{T}, \mathbf{Y}) \geq min(-\nu^2(\mathbf{S},\mathbf{Y}),-\nu^2(\mathbf{T},\mathbf{Y}))\end{equation}

 In the cases of ${\mathbf{S} \cap \mathbf{T}} \subset {\mathbf{S}}$ and  ${\mathbf{S} \cap \mathbf{T}} \subset {\mathbf{T}}$
the Kosorok's distance covariance inequality gives
\begin{equation}-\nu^2(\mathbf{S} \cap \mathbf{T}, \mathbf{Y}) > -\nu^2(\mathbf{S},\mathbf{Y})\end{equation} and
\begin{equation}-\nu^2(\mathbf{S} \cap \mathbf{T}, \mathbf{Y}) > -\nu^2(\mathbf{T},\mathbf{Y})\end{equation}
Thus, \begin{equation}-\nu^2(\mathbf{S} \cap \mathbf{T}, \mathbf{Y}) \geq min(-\nu^2(\mathbf{S},\mathbf{Y}),-\nu^2(\mathbf{T},\mathbf{Y}))\end{equation}

\end{proof}

\section{Conclusion} We showed that Algorithm 1 gives globally exact solutions that to the induced quasi-concave set function optimization and is highly parallelizable. This opens doors to a wide variety of real world applications that we would like to pursue as part of future work.
\nocite{langley00}

\bibliography{example_paper}
\bibliographystyle{icml2021}



\end{document}